\documentclass[onecolumn]{article}

\usepackage[a4paper, total={6in, 8.5in}]{geometry}
\usepackage{amsmath, amsfonts, amssymb, amsthm}
\usepackage{graphicx}
\usepackage{authblk}
\usepackage{titlesec}
\titleformat*{\section}{\large\bfseries}
\titleformat*{\subsection}{\large\bfseries}

\newtheorem{thm}{Theorem}[section]
\newtheorem{cor}[thm]{Corollary}
\newtheorem{lem}[thm]{Lemma}
\newtheorem{assum}[thm]{Assumption}
\newtheorem{prop}[thm]{Proposition}
\newtheorem{defn}[thm]{Definition}
\newtheorem{rem}[thm]{Remark}

\providecommand{\keywords}[1]
{
	\small	
	\textbf{\textit{Keywords---}} #1
}

\newcommand\blfootnote[1]{%
	\begingroup
	\renewcommand\thefootnote{}\footnote{#1}%
	\addtocounter{footnote}{-1}%
	\endgroup
}

\title{\textbf{Primal-dual evolutionary dynamics}\\
	   \textbf{for constrained population games}}
\author[1]{Juan Martinez-Piazuelo}
\author[1]{Nicanor Quijano}
\author[2]{Carlos Ocampo-Martinez}

\affil[1]{Departamento de Ingenier\'ia El\'ectrica y Electr\'onica, Universidad de los Andes, Carrera 1 No. 18A-10, Bogot\'a, Colombia}
\affil[2]{Automatic Control Department, Universitat Polit\`ecnica de Catalunya, Institut de Rob\`otica i Inform\`atica Industrial (CSIC-UPC), Llorens i Artigas, 4-6, 08028, Barcelona, Spain}

\date{}

\begin{document}

\maketitle

\begin{abstract}                          % Abstract of not more than 200 words.
	Population games can be regarded as a tool to study the strategic interaction of a population of players. Although several attention has been given to such field, most of the available works have focused only on the unconstrained case. That is, the allowed equilibrium of the game is not constrained. To further extend the capabilities of population games, in this paper we propose a novel class of primal-dual evolutionary dynamics that allow the consideration of constraints that must be satisfied at the equilibrium of the game. Using duality theory and Lyapunov stability theory, we provide sufficient conditions to guarantee the asymptotic stability and feasibility of the equilibria set of the game under the considered constraints. Furthermore, we illustrate the application of the developed theory to some classical population games with the addition of constraints.
\end{abstract}

\blfootnote{Corresponding author: Juan Martinez-Piazuelo (jp.martinez10@uniandes.edu.co)}
\blfootnote{Preprint. Under review.}

\section{Introduction}
Population games provide an evolutionary game theoretical framework to study the decision making process of a population of players \cite{hofbauer1998evolutionary}, \cite{sandholm2010}. As such, the study of population games and population dynamics have received significant attention over the control community \cite{quijano2017}. For instance, the authors in \cite{pashaie2017} illustrate the application of population dynamics to the dynamic resource allocation in a water distribution system. Similarly, the authors in \cite{tembine2010} rely on the formalism of population games to design evolutionary dynamics for the control of wireless networks. Furthermore, the authors in \cite{parkCDC2019} have extended the classic (memoryless) population games to a framework with dynamic payoff mechanisms, which allows the consideration of more general control scenarios, and the authors in \cite{fox2013population}, \cite{park2018CDC}, and \cite{arcak2020dissipativity}, have developed a set of powerful dissipativity tools for the study of such dynamical payoff models.

Although population games have been widely studied in the literature, most of the previous works have focused only on the unconstrained case. That is, the equilibrium of the game has no constraints regarding the amount of players playing the different strategies. Clearly, this is a significant limitation as real-world engineering applications usually have constraints over the control variables. One exception is the work of \cite{barreiroauto2016}, where the concept of mass dynamics is introduced to consider decoupled and coupled affine constraints that the players should asymptotically satisfy. Another exception is the work of \cite{barreiroauto2018}, where the authors propose some novel class of decision-making protocols so that a set of decoupled affine inequality constraints is dynamically satisfied. In contrast with such previous works, in this paper we consider general (coupled and/or decoupled) convex inequality constraints that the players should asymptotically satisfy at the equilibrium of the game. To the best of our knowledge, this is the first paper that studies such type of constraints in the context of population games.

Consequently, the main contribution of this paper is the formulation and analysis of a novel class of primal-dual evolutionary game dynamics, which allow the satisfaction of general convex inequality constraints at the equilibrium of the population game. Inspired by \cite{parkCDC2019}, and exploiting the duality theory of convex optimization, we propose a dynamic primal-dual game whose payoffs evolve over time to motivate the satisfaction of the given constraints. Moreover, using standard Lyapunov stability theory, we provide sufficient conditions to guarantee the asymptotic stability of the proposed dynamics for certain classes of constrained population games.

The remainder of this paper is organized as follows. Section \ref{sec:primal_dual} introduces the proposed primal-dual game and primal-dual evolutionary dynamics. Then, Section \ref{sec:analysis} provides the theoretical analyzes. Afterwards, Section \ref{sec:numerical_experiments} presents some numerical experiments as illustration. Finally, Section \ref{sec:concluding_remarks} concludes the paper and provides some future directions of research.

\section{Constrained population games and primal-dual evolutionary dynamics}
\label{sec:primal_dual}
Consider a population of players engaged in an anonymous game with a finite set of strategies denoted as $\mathcal{S}=\{1,2,\dots,n\}$, where $n\in\mathbb{Z}_{\geq2}$. Throughout, we refer to such population as the \textit{primal} population. At any time, the fraction of players playing the strategy $i\in\mathcal{S}$ is denoted as $x_i\in\mathbb{R}_{\geq0}$, and the state of the primal population is described by the vector $\mathbf{x}=[x_i]\in\mathbb{R}_{\geq0}^n$. Therefore, the set of all possible population states is given by
\begin{equation}
\label{eq:primal_population_set}
\Delta_{\mathcal{P}} = \left\{\mathbf{x}\in\mathbb{R}_{\geq0}^n\,:\, \sum_{i\in\mathcal{S}}x_i = m_{\mathcal{P}}\right\},
\end{equation}
where $m_\mathcal{P}\in\mathbb{R}_{>0}$ denotes the total mass of the primal population. Furthermore, every strategy $i\in\mathcal{S}$ has an associated fitness function, $f_i:\mathbb{R}_{\geq0}^n\rightarrow\mathbb{R}$, which provides the payoff obtained by the players playing such strategy at a given population state. Consequently, the primal population game is completely characterized by the fitness vector $\mathbf{f}:\mathbb{R}_{\geq0}^n\rightarrow\mathbb{R}^n$, which is obtained by assembling all the fitness functions into a vector, i.e., $\mathbf{f}(\cdot)=\left[f_i(\cdot)\right]\in\mathbb{R}^n$. Depending on the form of $\mathbf{f}(\cdot)$, a different type of game might be considered. In this paper, we focus mainly on the class of full-potential games.

\begin{defn}[\cite{sandholm2010}]
	\label{def:full_potential_games}
	Let $\mathbf{f}:\mathbb{R}_{\geq0}^n\rightarrow\mathbb{R}^n$ be a population game. If there exists a continuously differentiable (potential) function $p:\mathbb{R}_{\geq0}^n\rightarrow\mathbb{R}$ that satisfies $\nabla_{\mathbf{x}}p(\mathbf{x}) = \mathbf{f}(\mathbf{x})$ for all $\mathbf{x}\in\mathbb{R}_{\geq0}^n$, then $\mathbf{f}(\cdot)$ is a full-potential game.
\end{defn}

In the context of full-potential games, the goal for the population players is to collectively maximize a potential function. In contrast with most of the previous works where such maximization is unconstrained, in this paper the players must collectively maximize a potential function subject to constraints over the population state. More precisely, every equilibrium state of the primal game must be an optimal solution to the constrained optimization problem given by
\begin{equation}
\label{eq:primal_problem}
\begin{split}
\max_{\mathbf{x}\in\Delta_{\mathcal{P}}} \,\, p(\mathbf{x}) \quad \text{s.t.} \,\, g_k(\mathbf{x}) \leq 0, \quad \forall k\in\mathcal{C},
\end{split}
\end{equation}
where $\mathcal{C}=\{1,2,\dots,q\}$ denotes the set of constraints, with $|\mathcal{C}|=q\in\mathbb{Z}_{\geq1}$; and $g_k:\mathbb{R}_{\geq0}^n\rightarrow\mathbb{R}$ is a function that characterizes the $k$-th constraint. Throughout, we define the feasible region of the primal population game as
\begin{equation*}
\mathcal{X} = \left\{\mathbf{x}\in\mathbb{R}_{\geq0}^n\,:\,g_k(\mathbf{x}) \leq 0, \, \forall k\in\mathcal{C}\right\}.
\end{equation*}
Moreover, we often consider the following assumptions regarding the primal population game, the constraint functions $\{g_k(\cdot)\}_{k\in\mathcal{C}}$, and the feasible region $\mathcal{X}$.

\setcounter{thm}{0}
\begin{assum}
	\label{assump:full_potential_game}
	The primal population game $\mathbf{f}(\cdot)$ is a full-potential game with twice continuously differentiable concave potential function $p(\cdot)$.
\end{assum}

\begin{assum}
	\label{assump:convex_constraints}
	For every $k\in\mathcal{C}$, the corresponding constraint function $g_k(\cdot)$ is convex and twice continuously differentiable.
\end{assum}

\begin{assum}
	\label{assump:slater}
	There exists some $\mathbf{\tilde{x}}\in\mathbb{R}_{>0}^n\cap\Delta_\mathcal{P}$ such that $g_k(\mathbf{\tilde{x}})<0$ for all $k\in\mathcal{C}$.
\end{assum}

\setcounter{thm}{0}
\begin{rem}
	Notice that Assumptions \ref{assump:full_potential_game}, \ref{assump:convex_constraints}, and \ref{assump:slater} are common regularity conditions in the field of convex optimization. Namely, Assumptions \ref{assump:full_potential_game} and \ref{assump:convex_constraints} characterize the smoothness and convexity of the problem in (\ref{eq:primal_problem}), and Assumption \ref{assump:slater} is the Slater's constraint qualification condition which provides a sufficient condition for strong duality to hold \cite{bertsekas2009convex}.
\end{rem}

\begin{rem}
	Since a population game is a decentralized decision making process, i.e., there is not a global coordinator to set the players' strategies, and since we are interested in the consideration of general (decoupled and/or coupled) convex inequality constraints, in this paper we limit to the case where the constraints must be satisfied only at the equilibrium of the population game, i.e., asymptotically rather than dynamically. Note that such asymptotic satisfaction of constraints has also received significant attention in the related field of distributed optimization \cite{boyd2011distributed}, \cite{zhu2012}, \cite{LEI2016110}, \cite{liang2020}, \cite{ALGHUNAIM2020109003}.
\end{rem}

To handle the aforementioned constraints, in this paper we introduce a second population game. Namely, consider a second population of players, here referred to as the \textit{dual} population, where the players are engaged in a population game with a set of strategies denoted as $\mathcal{C}_e=\mathcal{C}\cup\{0\}$. Such set of strategies corresponds to the set of constraints $\mathcal{C}$ extended with an additional (null) strategy indexed by $0$. Hence, $|\mathcal{C}_e|=|\mathcal{C}|+1=q+1$. Throughout, we define $g_0:\mathbb{R}^n\rightarrow\{0\}$, i.e., $g_0(\cdot)=0$, such that the set $\{g_k(\cdot)\}_{k\in\mathcal{C}_e}$ is well defined. At any time, the fraction of players of the dual population playing the strategy $k\in\mathcal{C}_e$ is denoted as $\mu_k\in\mathbb{R}_{\geq0}$, and the state of the dual population is given by the vector $\boldsymbol{\mu}=\left[\mu_k\right]\in\mathbb{R}_{\geq0}^{q+1}$. Thus, the set of all possible dual population states is
\begin{equation}
\label{eq:dual_population_set}
\Delta_{\mathcal{D}} = \left\{\boldsymbol{\mu}\in\mathbb{R}_{\geq0}^{q+1}\,:\,\sum_{k\in\mathcal{C}_e}\mu_k=m_{\mathcal{D}}\right\},
\end{equation}
where $m_{\mathcal{D}}\in\mathbb{R}_{>0}$ is the total mass of the dual population. As before, every strategy $k\in\mathcal{C}_e$ has an associated function that provides the payoff obtained by the players playing such strategy. However, in contrast with the primal population game, in this case, such functions are independent of the dual population state, and, instead, take as argument the primal population state. More precisely, the payoff of the strategy $k\in\mathcal{C}_e$ at the primal state $\mathbf{x}\in\Delta_\mathcal{P}$ is given by $g_k(\mathbf{x})$. Therefore, the dual population game is fully characterized by the vector $\mathbf{g}\left(\cdot\right)=\left[g_k(\cdot)\right]\in\mathbb{R}^{q+1}$. In order to couple the primal game with the dual game, we further define the \textit{primal-dual} population game, $\mathbf{f}^\mu(\cdot, \cdot) = \left[f_i^\mu(\cdot, \cdot)\right]\in\mathbb{R}^n$, where $f_i^\mu:\mathbb{R}_{\geq0}^n\times\mathbb{R}_{\geq0}^{q+1}\rightarrow\mathbb{R}$, for all $i\in\mathcal{S}$. Namely, the value $f_i^\mu(\mathbf{x}, \boldsymbol{\mu})$ provides the payoff obtained by the players of the primal population playing the strategy $i\in\mathcal{S}$ at the primal state $\mathbf{x}\in\Delta_\mathcal{P}$ and at the dual state $\boldsymbol{\mu}\in\Delta_\mathcal{D}$. In particular, in this paper we set
\begin{equation}
\label{eq:primal_dual_game}
f_i^\mu(\mathbf{x}, \boldsymbol{\mu}) = f_i(\mathbf{x}) - \sum_{k\in\mathcal{C}_e}\mu_k\frac{\partial g_k(\mathbf{x})}{\partial x_i}, \quad \forall i\in\mathcal{S}.
\end{equation}
Thus, the primal-dual game is an extension of the primal game, $\mathbf{f}(\cdot)$, that considers the constraints $\{g_k(\cdot)\}_{k\in\mathcal{C}}$ and the dual state $\boldsymbol{\mu}\in\Delta_\mathcal{D}$.

Now that we have defined the primal, dual, and primal-dual games, we proceed to introduce the evolutionary dynamics that describe the evolution of the primal and dual population states. Following the approach of \cite{sandholm2010}, we assume that the players of each population are equipped with some revision protocols to revise their strategies. More precisely, in this paper we focus on the class of impartial pairwise comparison protocols \cite{parkCDC2019}. Namely, the players of the primal population are equipped with a protocol $\rho_{j}:\mathbb{R}\rightarrow\mathbb{R}_{\geq0}$, for all $j\in\mathcal{S}$, which provides the incentive to switch to strategy $j\in\mathcal{S}$, and the players of the dual population are equipped with a protocol $\phi_{l}:\mathbb{R}\rightarrow\mathbb{R}_{\geq0}$, for all $l\in\mathcal{C}_e$, which provides the incentive to switch to strategy $l\in\mathcal{C}_e$. Moreover, for every $j\in\mathcal{S}$ and every $l\in\mathcal{C}_e$, the functions $\rho_j(\cdot)$ and $\phi_l(\cdot)$ are locally Lipschitz continuous and satisfy that
\begin{subequations}
	\label{eq:conditions_protocols}
	\begin{align}
	&\label{eq:conditions_protocols_primal}\left\{\begin{array}{cc} \rho_j(\alpha) > 0, &\text{if }\alpha>0\\
	\rho_j(\alpha) = 0, &\text{if }\alpha\leq 0\end{array}\right. \quad \forall \alpha\in\mathbb{R},\,\forall j\in\mathcal{S},\\
	&\label{eq:conditions_protocols_dual}\left\{\begin{array}{cc} \phi_l(\alpha) > 0, &\text{if }\alpha>0\\
	\phi_l(\alpha) = 0, &\text{if }\alpha\leq 0\end{array}\right. \quad \forall \alpha\in\mathbb{R},\,\forall l\in\mathcal{C}_e.
	\end{align}
\end{subequations}
Therefore, by letting $\rho_i^j:\mathbb{R}_{\geq0}^n\times\mathbb{R}_{\geq0}^{q+1}\rightarrow\mathbb{R}_{\geq0}$ denote the incentive to switch from strategy $i\in\mathcal{S}$ to strategy $j\in\mathcal{S}$, the primal evolutionary dynamics are stated as
\begin{subequations}
	\label{eq:primal_dynamics}
	\begin{align}
	&\label{eq:primal_memoryless}\rho_i^j(\mathbf{x}, \boldsymbol{\mu}) = \rho_j\left(f_j^\mu(\mathbf{x}, \boldsymbol{\mu}) - f_i^\mu(\mathbf{x}, \boldsymbol{\mu})\right), \quad \forall i,j\in\mathcal{S},\\
	&\label{eq:primal_dot}\dot{x}_i = \sum_{j\in\mathcal{S}}\left(x_j\rho_j^i(\mathbf{x}, \boldsymbol{\mu}) - x_i\rho_i^j(\mathbf{x}, \boldsymbol{\mu})\right), \quad \forall i\in\mathcal{S}.
	\end{align}
\end{subequations}
Similarly, by letting $\phi_k^l:\mathbb{R}_{\geq0}^n\rightarrow\mathbb{R}_{\geq0}$ denote the incentive to switch from strategy $k\in\mathcal{C}_e$ to strategy $l\in\mathcal{C}_e$, the dual evolutionary dynamics are defined as
\begin{subequations}
	\label{eq:dual_dynamics}
	\begin{align}
	&\label{eq:dual_memoryless}\phi_k^l(\mathbf{x}) = \phi_l\left(g_l(\mathbf{x}) - g_k(\mathbf{x})\right), \quad \forall k,l\in\mathcal{C}_e,\\
	&\label{eq:dual_dot}\dot{\mu}_k = \sum_{l\in\mathcal{C}_e}\left(\mu_l\phi_l^k(\mathbf{x}) - \mu_k\phi_k^l(\mathbf{x})\right), \quad \forall k\in\mathcal{C}_e.
	\end{align}
\end{subequations}

\begin{rem}
	\label{rem:edm_pdm}
	Notice that, as shown in Fig. \ref{fig:primal_dual_dynamics}, the primal-dual game (\ref{eq:primal_dual_game}) and the primal and dual dynamics (\ref{eq:primal_dynamics})-(\ref{eq:dual_dynamics}) can be seen as two different nonlinear dynamical systems that are interconnected in a positive feedback loop. Furthermore, such a positive feedback loop interconnection of evolutionary dynamics can be regarded as an instance of the EDM-PDM (evolutionary dynamics model - payoff dynamics model) systems proposed in \cite{parkCDC2019}, \cite{parkarxiv2019}. For instance, the primal-dual game (\ref{eq:primal_dual_game}) and the primal dynamics (\ref{eq:primal_dynamics}) can be thought as an EDM that describes the evolution of the primal population state as a function of the dual state $\boldsymbol{\mu}\in\Delta_\mathcal{D}$, and the dual dynamics (\ref{eq:dual_dynamics}) can be seen as a PDM that dynamically modifies the payoffs for the EDM based on the primal state $\mathbf{x}\in\Delta_\mathcal{P}$.
\end{rem}

\begin{rem}
	As mentioned above, in this paper we focus on the class of evolutionary dynamics that result from using impartial pairwise comparison protocols. The motivation being that such a class of dynamics has some desirable properties that we exploit in our theoretical analyses. Nevertheless, it is worth noting that in the literature there are many other classes of evolutionary dynamics that might be worth exploring under the considered primal-dual framework. Some novel examples include generalized imitation dynamics \cite{zino2017}, relative best response dynamics \cite{govaert2019}, and proximal dynamics \cite{grammatico2018}, among others.
\end{rem}

As shown in Fig. \ref{fig:primal_dual_dynamics} and highlighted in Remark \ref{rem:edm_pdm}, the primal-dual game, the primal dynamics, and the dual dynamics, together can be thought as the components of two nonlinear dynamical systems that interact in a positive feedback loop interconnection. Throughout, we refer to such an interconnected system as the primal-dual system, which is an $(n+q+1)$-dimensional system whose state vector is given by $[\mathbf{x}^\top, \boldsymbol{\mu}^\top]^\top\in\mathbb{R}_{\geq0}^{(n+q+1)}$. Moreover, we further impose the following standing assumption regarding the initial conditions of that system.\\

\noindent
\textbf{Standing Assumption 1}\, $\mathbf{x}(0)\in\Delta_\mathcal{P}$ \textit{and} $\boldsymbol{\mu}(0)\in\Delta_\mathcal{D}$.\\

In the forthcoming section, we provide the theoretical analyses regarding such a primal-dual system.

\begin{figure}
	\centering
	\includegraphics[width=0.45\textwidth]{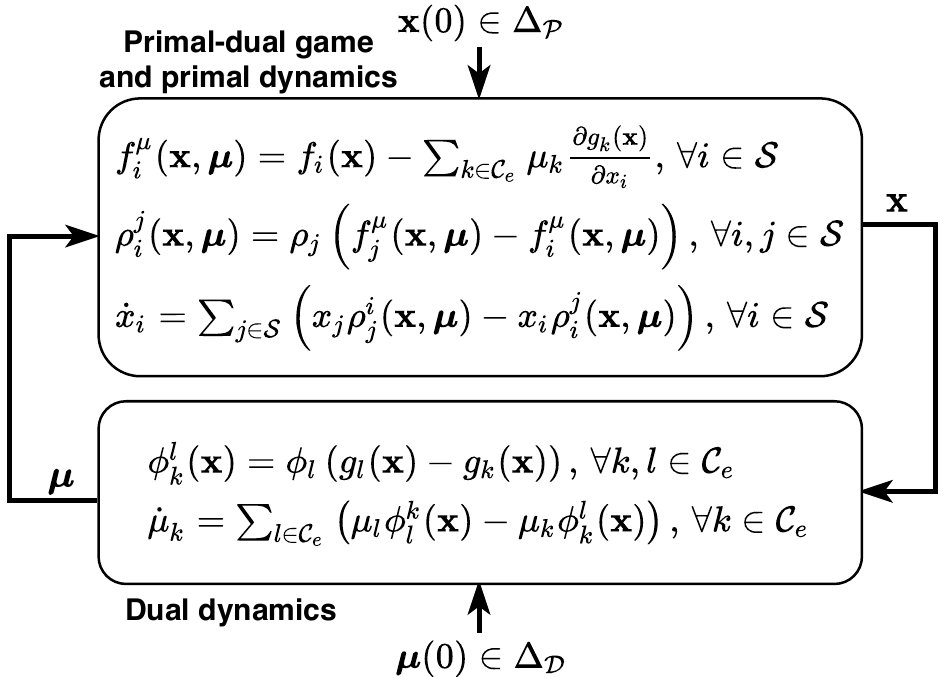}
	\caption{\,Positive feedback loop interconnection of the primal-dual game, the primal dynamics, and the dual dynamics.}
	\label{fig:primal_dual_dynamics}
\end{figure}

\section{Analysis of the proposed primal-dual system}
\label{sec:analysis}
In this section, we provide our main theoretical developments regarding the primal-dual system presented in Section \ref{sec:primal_dual}. First, we show some invariance properties of the considered dynamics. Then, we characterize the equilibria set of the primal-dual system. Finally, we provide our main results regarding the stability of the interconnected nonlinear system.

\subsection{Invariance analysis}
In this section, we characterize some sets that are positively invariant under the considered primal-dual system. Namely, $\Delta_\mathcal{P}$ is positively invariant under the dynamics (\ref{eq:primal_dynamics}), and $\Delta_\mathcal{D}$ is positively invariant under the dynamics (\ref{eq:dual_dynamics}). More precisely, $\mathbf{x}(0)\in\Delta_\mathcal{P} \implies \mathbf{x}(t)\in\Delta_\mathcal{P}$, for all $t\geq0$, and $\boldsymbol{\mu}(0)\in\Delta_\mathcal{D} \implies \boldsymbol{\mu}(t)\in\Delta_\mathcal{D}$, for all $t\geq0$. These invariance properties are formally stated in Lemmas \ref{lem:invariance_primal} and \ref{lem:invariance_dual}, respectively.

\setcounter{thm}{0}
\begin{lem}
	\label{lem:invariance_primal}
	The set $\Delta_\mathcal{P}$ is positively invariant under the primal dynamics (\ref{eq:primal_dynamics}).
\end{lem}
\begin{proof}
	First note that
	\begin{equation*}
	\begin{split}
	\sum_{i\in\mathcal{S}}\dot{x}_i &= \sum_{i\in\mathcal{S}}\sum_{j\in\mathcal{S}}\left(x_j\rho_j^i(\mathbf{x}, \boldsymbol{\mu}) - x_i\rho_i^j(\mathbf{x}, \boldsymbol{\mu})\right)\\
	&= \sum_{i\in\mathcal{S}}\sum_{j\in\mathcal{S}}x_j\rho_j^i(\mathbf{x}, \boldsymbol{\mu}) -\sum_{i\in\mathcal{S}}\sum_{j\in\mathcal{S}} x_i\rho_i^j(\mathbf{x}, \boldsymbol{\mu})\\
	&= \sum_{i\in\mathcal{S}}\sum_{j\in\mathcal{S}}x_j\rho_j^i(\mathbf{x}, \boldsymbol{\mu}) -\sum_{i\in\mathcal{S}}\sum_{j\in\mathcal{S}} x_j\rho_j^i(\mathbf{x}, \boldsymbol{\mu})\\
	&= 0.
	\end{split}
	\end{equation*}
	Thus, $\sum_{i\in\mathcal{S}}x_i(0) = m_\mathcal{P} \implies \sum_{i\in\mathcal{S}}x_i(t) = m_\mathcal{P}$, for all $t\geq0$. Second, for every $i\in\mathcal{S}$ it holds that if $x_i=0$, then $\dot{x}_i\geq0$. Hence, $\mathbf{x}(0)\in\mathbb{R}_{\geq0}^n \implies \mathbf{x}(t)\in\mathbb{R}_{\geq0}^n$, for all $t\geq0$. Therefore, $\mathbf{x}(0)\in\Delta_\mathcal{P} \implies \mathbf{x}(t)\in\Delta_\mathcal{P}, \,\forall t\geq0$.%\hfill $\blacksquare$
\end{proof}

\begin{lem}
	\label{lem:invariance_dual}
	The set $\Delta_\mathcal{D}$ is positively invariant under the dual dynamics (\ref{eq:dual_dynamics}).
\end{lem}
\begin{proof}
	Due to the similarity between (\ref{eq:primal_dynamics}) and (\ref{eq:dual_dynamics}), the proof is virtually identical to the proof of Lemma \ref{lem:invariance_primal}.%\hfill $\blacksquare$
\end{proof}

Furthermore, Lemmas \ref{lem:invariance_primal} and \ref{lem:invariance_dual} together lead to the following result that fully characterizes the invariance properties of the proposed primal-dual system.

\setcounter{thm}{0}
\begin{prop}
	\label{prop:invariance}
	The set $\Delta_\mathcal{P}\times\Delta_\mathcal{D}\subset\mathbb{R}_{\geq0}^{n+q+1}$ is positively invariant under the proposed primal-dual system.
\end{prop}
\begin{proof}
	The result follows immediately from Lemmas \ref{lem:invariance_primal} and \ref{lem:invariance_dual}. %\hfill $\blacksquare$
\end{proof}

\setcounter{thm}{4}
\begin{rem}
	Note that the Standing Assumption 1, in conjunction with Proposition \ref{prop:invariance}, allows us to assume, without any additional loss of generality, that $\mathbf{x}(t)\in\Delta_\mathcal{P}$ and $\boldsymbol{\mu}(t)\in\Delta_\mathcal{D}$, for all $t\geq0$. This fact plays a crucial role in the forthcoming analyses.
\end{rem}

\subsection{Equilibria set analysis}
We now proceed to characterize the equilibria set of the considered primal-dual system. For such, let us first introduce the concepts of Nash equilibria for the primal, dual, and primal-dual games.

Typically, the set of Nash equilibria of a population game is defined as the set of population states where no player has incentives to change her strategy \cite{sandholm2010}. With this concept in mind, the set of Nash equlibria of the primal game $\mathbf{f}(\cdot)$ can be defined as
\begin{equation}
\label{eq:nash_primal}
\text{NE}\left(\mathbf{f}\right) = \left\{\mathbf{x}\in\Delta_\mathcal{P}\,:\,\mathbf{x}\in\text{arg}\max_{\mathbf{y}\in\Delta_\mathcal{P}}\mathbf{y}^\top\mathbf{f}(\mathbf{x})\right\}.
\end{equation}
In contrast, the Nash equilibria of the dual game $\mathbf{g}(\cdot)$ is a function of the primal state $\mathbf{x}\in\Delta_\mathcal{P}$. More precisely,
\begin{equation}
\label{eq:nash_dual}
\text{NE}\left(\mathbf{g}, \mathbf{x}\right) = \left\{\boldsymbol{\mu}\in\Delta_\mathcal{D}\,:\,\boldsymbol{\mu}\in\text{arg}\max_{\mathbf{z}\in\Delta_\mathcal{D}}\mathbf{z}^\top\mathbf{g}(\mathbf{x})\right\}.
\end{equation}
Similarly, the Nash equilibria of the primal-dual game $\mathbf{f}^\mu(\cdot, \cdot)$ is a function of the dual state $\boldsymbol{\mu}\in\Delta_\mathcal{D}$. Namely,
\begin{equation}
\label{eq:nash_primal_dual}
\text{NE}\left(\mathbf{f}^\mu, \boldsymbol{\mu}\right) = \left\{\mathbf{x}\in\Delta_\mathcal{P}\,:\,\mathbf{x}\in\text{arg}\max_{\mathbf{y}\in\Delta_\mathcal{P}}\mathbf{y}^\top\mathbf{f}^\mu(\mathbf{x}, \boldsymbol{\mu})\right\}.
\end{equation}
Therefore, for any fixed $\boldsymbol{\mu}^*\in\Delta_\mathcal{D}$, the set $\text{NE}\left(\mathbf{f}^\mu, \boldsymbol{\mu}^*\right)$ contains the Nash equilibria, in the sense of (\ref{eq:nash_primal}), of the population game $\mathbf{f}^\mu(\cdot, \boldsymbol{\mu}^*)$. These definitions for the set of Nash equilibria of the dual and primal-dual games allow us to provide the following results.

\setcounter{thm}{2}
\begin{lem}
	\label{lem:dual_nash_stationarity}
	Suppose that $\mathbf{g}(\cdot)$ is continuous. Consider the dual dynamics (\ref{eq:dual_dynamics}), let $\mathbf{\dot{\boldsymbol{\mu}}}=\left[\dot{\mu}_k\right]\in\mathbb{R}^{q+1}$, and let $\mathbf{x}^*\in\Delta_\mathcal{P}$. Then, $\text{\normalfont NE}\left(\mathbf{g}, \mathbf{x}^*\right)$ is nonempty, and $\mathbf{\dot{\boldsymbol{\mu}}}=\mathbf{0}$ if and only if $\boldsymbol{\mu}\in\text{\normalfont NE}\left(\mathbf{g}, \mathbf{x}^*\right)$.
\end{lem}
\begin{proof}
	The claim that $\text{NE}\left(\mathbf{g}, \mathbf{x}^*\right)\neq\emptyset$ follows from the continuity of $\mathbf{g}(\cdot)$, the compactness of $\Delta_\mathcal{D}$, and the Weierstrass Theorem. To prove the second claim, let us consider the sufficient and necessary cases separately.
	
	(Sufficiency) Let $\boldsymbol{\mu}\in\text{NE}\left(\mathbf{g},\mathbf{x}^*\right)$. From (\ref{eq:nash_dual}), it holds that $\mu_k>0\implies g_k(\mathbf{x}^*)=\max_{l\in\mathcal{C}_e}g_l(\mathbf{x}^*)$, for all $k\in\mathcal{C}_e$. Thus, from (\ref{eq:conditions_protocols_dual}) and (\ref{eq:dual_memoryless}) it follows that $\mu_k\phi_k^l(\mathbf{x}^*)=0$, for all $k,l\in\mathcal{C}_e$. Hence, $\mathbf{\dot{\boldsymbol{\mu}}}=\mathbf{0}$.
	
	(Necessity) Suppose that $\mathbf{\dot{\boldsymbol{\mu}}}=\mathbf{0}$ but $\boldsymbol{\mu}\notin\text{NE}\left(\mathbf{g}, \mathbf{x}^*\right)$. Let $k\in\mathcal{C}_e$ be such that $g_k(\mathbf{x}^*)=\max_{l\in\mathcal{C}_e}g_l(\mathbf{x}^*)$. Thus, from (\ref{eq:conditions_protocols_dual}) and (\ref{eq:dual_memoryless}) it follows that $\mu_k\phi_k^l(\mathbf{x}^*) = 0$, for all $l\in\mathcal{C}_e$. Hence, $\dot{\mu}_k\geq0$. Now, since $\boldsymbol{\mu}\notin\text{NE}\left(\mathbf{g}, \mathbf{x}^*\right)$, there exists some $z\in\mathcal{C}_e$ such that $\mu_z>0$ and $g_z(\mathbf{x}^*) < g_k(\mathbf{x}^*)$. Therefore, $x_z\phi_z^k(\mathbf{x}^*)>0$ and $\dot{\mu}_k>0$. In consequence, $\mathbf{\dot{\boldsymbol{\mu}}}\neq\mathbf{0}$, which is a contradiction. %\hfill $\blacksquare$
\end{proof}

\begin{lem}
	\label{lem:primal_dual_nash_stationarity}
	Suppose that $\mathbf{f}^\mu(\cdot, \cdot)$ is continuous. Consider the primal dynamics (\ref{eq:primal_dynamics}), let $\mathbf{\dot{x}}=\left[\dot{x}_i\right]\in\mathbb{R}^{n}$, and let $\boldsymbol{\mu}^*\in\Delta_\mathcal{D}$. Then, $\text{\normalfont NE}\left(\mathbf{f}^\mu, \boldsymbol{\mu}^*\right)$ is nonempty, and $\mathbf{\dot{x}}=\mathbf{0}$ if and only if $\mathbf{x}\in\text{\normalfont NE}\left(\mathbf{f}^\mu, \boldsymbol{\mu}^*\right)$.
\end{lem}
\begin{proof}
	Noting from (\ref{eq:nash_primal_dual}) that $\mathbf{x}\in\text{NE}\left(\mathbf{f}^\mu, \boldsymbol{\mu}^*\right)$ implies that $x_i>0\implies f_i^\mu(\mathbf{x}, \boldsymbol{\mu}^*)=\max_{j\in\mathcal{S}}f_j^\mu(\mathbf{x}, \boldsymbol{\mu}^*)$, for all $i\in\mathcal{S}$, the proof is virtually identical to the one of Lemma \ref{lem:dual_nash_stationarity}.%\hfill $\blacksquare$
\end{proof}

Lemmas \ref{lem:dual_nash_stationarity} and \ref{lem:primal_dual_nash_stationarity} characterize a property of the considered dual and primal dynamics termed as Nash stationarity \cite{sandholm2010}. Namely, the equilibria set of the considered dual dynamics coincides with the set of Nash equilibria of the dual game $\mathbf{g}(\cdot)$, whilst the equilibria set of the primal dynamics coincides with the set of Nash equilibria of the primal-dual game $\mathbf{f}^\mu(\cdot, \cdot)$. Consequently, Lemmas \ref{lem:dual_nash_stationarity} and \ref{lem:primal_dual_nash_stationarity} allow us to fully characterize the equilibria set of the considered primal-dual system.

\setcounter{thm}{0}
\begin{thm}
	\label{thm:equilibria_set_definition}
	Consider the primal-dual system (\ref{eq:primal_dynamics})-(\ref{eq:dual_dynamics}), and suppose that $\mathbf{g}(\cdot)$ and $\mathbf{f}^\mu(\cdot, \cdot)$ are continuous. Then, a point $(\mathbf{x}^*, \boldsymbol{\mu}^*)\in\Delta_{\mathcal{P}}\times\Delta_{\mathcal{D}}$ is an equilibrium state of the primal-dual system if and only if $(\mathbf{x}^*, \boldsymbol{\mu}^*)\in\mathcal{E}$, where
	\begin{equation}
	\label{eq:equilibria_set}
	\mathcal{E} = \left\{(\mathbf{x}, \boldsymbol{\mu})\in\Delta_\mathcal{P}\times\Delta_\mathcal{D}\,:\,\begin{array}{c}\mathbf{x}\in\text{\normalfont NE}\left(\mathbf{f}^\mu, \boldsymbol{\mu}\right)\\
	\boldsymbol{\mu}\in\text{\normalfont NE}\left(\mathbf{g}, \mathbf{x}\right)
	\end{array}\right\}.
	\end{equation}
\end{thm}

\begin{proof}
	The result follows immediately from Lemmas \ref{lem:dual_nash_stationarity} and \ref{lem:primal_dual_nash_stationarity}.%\hfill $\blacksquare$
\end{proof}

Although Theorem \ref{thm:equilibria_set_definition} provides necessary and sufficient conditions for a state $\left[\mathbf{x}^\top,\boldsymbol{\mu}^\top\right]^\top\in\mathbb{R}_{\geq0}^{n+q+1}$ to be an equilibrium point of the primal-dual system, it does not guarantee the existence of such an equilibrium state. Moreover, Theorem \ref{thm:equilibria_set_definition} by itself does not guarantee that the equilibria set $\mathcal{E}$ satisfies the considered constraints $\{g_k(\cdot)\}_{k\in\mathcal{C}}$. To address both of these issues, we provide the following results.

\setcounter{thm}{4}
\begin{lem}
	\label{lem:lagrangian_saddle_points}
	Let Assumptions \ref{assump:full_potential_game} and \ref{assump:convex_constraints} hold. Consider the primal-dual system (\ref{eq:primal_dynamics})-(\ref{eq:dual_dynamics}), the equilibria set $\mathcal{E}$ in (\ref{eq:equilibria_set}), and the function
	\begin{equation}
	\label{eq:lagrangian_extended}
	L(\mathbf{x}, \boldsymbol{\mu}) = p(\mathbf{x}) - \sum_{k\in\mathcal{C}_e}\mu_kg_k(\mathbf{x}).
	\end{equation}
	Then, $\mathcal{E}$ coincides with the set of saddle points of (\ref{eq:lagrangian_extended}). More precisely, $(\mathbf{x}^*, \boldsymbol{\mu}^*)\in\mathcal{E}$ if and only if
	\begin{equation*}
	L(\mathbf{x}, \boldsymbol{\mu}^*) \leq L(\mathbf{x}^*, \boldsymbol{\mu}^*) \leq L(\mathbf{x}^*, \boldsymbol{\mu}), \,\,\, \forall \mathbf{x}\in\Delta_{\mathcal{P}}, \,\forall\boldsymbol{\mu}\in\Delta_{\mathcal{D}}.
	\end{equation*}
\end{lem}

\begin{proof}
	Note that $\nabla_{\mathbf{x}}L(\mathbf{x}, \boldsymbol{\mu}) = \mathbf{f}^\mu(\mathbf{x}, \boldsymbol{\mu})$, and $\nabla_{\boldsymbol{\mu}}L(\mathbf{x}, \boldsymbol{\mu}) = -\mathbf{g}(\mathbf{x})$. Thus, from Definition \ref{def:full_potential_games}, for every fixed $\boldsymbol{\mu}^*\in\Delta_{\mathcal{D}}$, the game $\mathbf{f}^\mu(\cdot, \boldsymbol{\mu}^*)$ is a full-potential game over $\Delta_\mathcal{P}$ with concave potential function $L(\cdot, \boldsymbol{\mu}^*)$. Similarly, for every fixed $\mathbf{x}^*\in\Delta_{\mathcal{P}}$, the game $\mathbf{g}(\mathbf{x}^*)$ is a full-potential game over $\Delta_\mathcal{D}$ with affine (and thus concave) potential function $-L(\mathbf{x}^*, \cdot)$. Therefore, using \cite[Corollary 3.1.4]{sandholm2010} we conclude that $\mathbf{x}^*\in\text{NE}\left(\mathbf{f}^\mu, \boldsymbol{\mu}^*\right)$ if and only if $\mathbf{x}^*\in\text{arg}\max_{\mathbf{x}\in\Delta_\mathcal{P}}L(\cdot, \boldsymbol{\mu}^*)$, and $\boldsymbol{\mu}^*\in\text{NE}\left(\mathbf{g}, \mathbf{x}^*\right)$ if and only if $\boldsymbol{\mu}^*\in\text{arg}\min_{\boldsymbol{\mu}\in\Delta_\mathcal{D}}L(\mathbf{x}^*, \cdot)$. Hence,
	\begin{equation*}
	\begin{split}
	&\mathbf{x}^*\in\text{NE}\left(\mathbf{f}^\mu, \boldsymbol{\mu}^*\right) \iff L(\mathbf{x}, \boldsymbol{\mu}^*) \leq L(\mathbf{x}^*, \boldsymbol{\mu}^*), \,\forall \mathbf{x}\in\Delta_{\mathcal{P}},\\
	&\boldsymbol{\mu}^*\in\text{NE}\left(\mathbf{g}, \mathbf{x}^*\right) \iff L(\mathbf{x}^*, \boldsymbol{\mu}^*) \leq L(\mathbf{x}^*, \boldsymbol{\mu}), \,\,\forall \boldsymbol{\mu}\in\Delta_{\mathcal{D}},\\
	\end{split}
	\end{equation*}
	which leads to the desired result. %\hfill $\blacksquare$
\end{proof}

\setcounter{thm}{1}
\begin{thm}
	\label{thm:equilibria_set_properties}
	Let Assumptions \ref{assump:full_potential_game}, \ref{assump:convex_constraints}, and \ref{assump:slater} hold. Consider the primal-dual system (\ref{eq:primal_dynamics})-(\ref{eq:dual_dynamics}), and the equilibria set $\mathcal{E}$ in (\ref{eq:equilibria_set}). If the mass of the dual population satisfies that
	\begin{equation}
	\label{eq:condition_dual_population_mass}
	m_\mathcal{D} \geq \frac{p^* - p(\mathbf{\tilde{x}})}{\min_{k\in\mathcal{C}}\left|g_k(\mathbf{\tilde{x}})\right|},
	\end{equation}
	where $p^*\triangleq\max_{\mathbf{x}\in\Delta_\mathcal{P}\cap\mathcal{X}}p(\mathbf{x})$, and $\mathbf{\tilde{x}}\in\mathbb{R}_{>0}\cap\Delta_\mathcal{P}$ is any vector satisfying Assumption \ref{assump:slater}, then the set $\mathcal{E}$ is nonempty and every point $(\mathbf{x}^*, \boldsymbol{\mu}^*)\in\mathcal{E}$ satisfies that $\mathbf{x}^*\in\text{\normalfont arg}\max_{\mathbf{x}\in\Delta_\mathcal{P}\cap\mathcal{X}}p(\mathbf{x})$, i.e., $\mathbf{x}^*$ solves (\ref{eq:primal_problem}).
\end{thm}

\begin{proof}
	To prove this result, let us consider a convex optimization perspective. Consider the problem in (\ref{eq:primal_problem}) and note that its feasible set is $\Delta_\mathcal{P}\cap\mathcal{X}$. From Assumption \ref{assump:convex_constraints}, it follows that $\mathcal{X}$ is a closed set (it is the intersection of the level sets of continuous functions, which are all closed sets). Hence, since $\Delta_\mathcal{P}$ is compact, it follows that $\Delta_\mathcal{P}\cap\mathcal{X}$ is compact. Therefore, since $p(\cdot)$ is continuous over $\Delta_\mathcal{P}\cap\mathcal{X}$, from the Weierstrass Theorem we conclude that there exists an $\mathbf{x}^*\in\Delta_\mathcal{P}\cap\mathcal{X}$ such that $\mathbf{x}^*\in\text{arg}\max_{\mathbf{x}\in\Delta_\mathcal{P}\cap\mathcal{X}}p(\mathbf{x})$ and $p^*$ is finite, i.e., the set of optimal solutions of (\ref{eq:primal_problem}) is nonempty and the optimal value is finite. Using this observation, in conjunction with Assumption \ref{assump:slater} (Slater's condition), we conclude that there exists a vector $\boldsymbol{\lambda}^*\in\mathbb{R}_{\geq0}^q$ such that $p^*=d^*\triangleq d(\boldsymbol{\lambda}^*)$, where
	\begin{equation*}
	d(\boldsymbol{\lambda}) = \max_{\mathbf{x}\in\Delta_\mathcal{P}}\left(p(\mathbf{x}) - \sum_{k\in\mathcal{C}}\lambda_kg_k(\mathbf{x})\right)
	\end{equation*}
	is the dual function of the dual problem of (\ref{eq:primal_problem}), i.e., $d^* = \min_{\boldsymbol{\lambda}\in\mathbb{R}_{\geq0}^q}d(\boldsymbol{\lambda})$. Namely, strong duality holds for the problem in (\ref{eq:primal_problem}) and there exists a primal-dual optimal solution $(\mathbf{x}^*, \boldsymbol{\lambda}^*)$. Furthermore, note that if $\mathbf{\tilde{x}}$ is any vector satisfying Assumption \ref{assump:slater}, then it follows that
	\begin{equation*}
	\begin{split}
	d^* = p^* &= \max_{\mathbf{x}\in\Delta_\mathcal{P}}\left(p(\mathbf{x}) - \sum_{k\in\mathcal{C}}\lambda_k^*g_k(\mathbf{x})\right)\\
	&\geq p(\mathbf{\tilde{x}}) - \sum_{k\in\mathcal{C}}\lambda_k^*g_k(\mathbf{\tilde{x}})\\
	&\geq p(\mathbf{\tilde{x}}) + \min_{z\in\mathcal{C}}\left|g_z(\mathbf{\tilde{x}})\right|\sum_{k\in\mathcal{C}}\lambda_k^*\quad\text{(since $\mathbf{\tilde{x}}\in\mathcal{X}$)}.
	\end{split}
	\end{equation*}
	Therefore, since $\min_{z\in\mathcal{C}}\left|g_z(\mathbf{\tilde{x}})\right|>0$, it follows that
	\begin{equation*}
	\sum_{k\in\mathcal{C}}\lambda_k^* \leq \frac{p^* - p(\mathbf{\tilde{x}})}{\min_{z\in\mathcal{C}}\left|g_z(\mathbf{\tilde{x}})\right|}\in[0, \infty).
	\end{equation*}
	In consequence, the set of dual optimal solutions of the dual of (\ref{eq:primal_problem}) is bounded. Now, from Lemma \ref{lem:invariance_dual} it holds that $\mu_0(t) = m_\mathcal{D} - \sum_{k\in\mathcal{C}}\mu_k(t)$, for all $t\geq0$. Hence, if we let $\mu_k^*=\lambda_k^*$, for all $k\in\mathcal{C}$, then the vector $\boldsymbol{\mu}^*\in\Delta_{\mathcal{D}}\subset\mathbb{R}_{\geq0}^{q+1}$ is fully determined by the vector $\boldsymbol{\lambda}^*\in\mathbb{R}_{\geq0}^q$ and the mass $m_\mathcal{D}$. Therefore, if $m_\mathcal{D}$ satisfies (\ref{eq:condition_dual_population_mass}), then there exists a point $(\mathbf{x}^*, \boldsymbol{\mu}^*)\in\Delta_\mathcal{P}\times\Delta_\mathcal{D}$ that has direct correspondence with the primal-dual optimal solution $(\mathbf{x}^*, \boldsymbol{\lambda}^*)$. More precisely, the set of primal-dual optimal solutions of (\ref{eq:primal_problem}) is attainable under the proposed primal-dual evolutionary dynamics. Now, note that since $g_0(\cdot)=0$, it follows that (\ref{eq:lagrangian_extended}) is the dual Lagrangian function of (\ref{eq:primal_problem}), where $\mu_k\in\mathbb{R}_{\geq0}$ is the Lagrange multiplier of the $k$-th inequality constraint. Using Lemma \ref{lem:lagrangian_saddle_points} and \cite[Proposition 3.4.1]{bertsekas2009convex}, we conclude that $(\mathbf{x}^*, \boldsymbol{\mu}^*)\in\mathcal{E}$ if and only if $(\mathbf{x}^*, \boldsymbol{\mu}^*)$ corresponds to a primal-dual optimal solution $(\mathbf{x}^*, \boldsymbol{\lambda}^*)$. Therefore, from the previous discussion it holds that $\mathcal{E}$ is nonempty and that every point $(\mathbf{x}^*, \boldsymbol{\mu}^*)\in\mathcal{E}$ satisfies that $\mathbf{x}^*\in\text{\normalfont arg}\max_{\mathbf{x}\in\Delta_\mathcal{P}\cap\mathcal{X}}p(\mathbf{x})$. %\hfill $\blacksquare$
\end{proof}

\setcounter{thm}{5}
\begin{rem}
	\label{rem:feasible_with_respect_to_X}
	Note that Theorem \ref{thm:equilibria_set_properties} provides sufficient conditions to guarantee that the equilibria set $\mathcal{E}$ is nonempty and feasible with respect to $\mathcal{X}$, i.e., any $(\mathbf{x}^*, \boldsymbol{\mu}^*)\in\mathcal{E}$ satisfies that $\mathbf{x}^*\in\mathcal{X}$. In particular, notice that if the right hand side of (\ref{eq:condition_dual_population_mass}) is zero, then the primal optimal $\mathbf{x}^*$ lies in the interior of $\mathcal{X}$, and, in consequence, the dual game is not required for the satisfaction of the constraints.
\end{rem}

\subsection{Lyapunov stability analysis}
In this section, we provide our main results on the asymptotic stability of the proposed primal-dual system. For such, we develop the corresponding Lyapunov stability analysis of the nonlinear primal-dual evolutionary dynamics in (\ref{eq:primal_dynamics})-(\ref{eq:dual_dynamics}).

\setcounter{thm}{5}
\begin{lem}
	\label{lem:lyapunov_function}
	Consider the primal-dual system (\ref{eq:primal_dynamics})-(\ref{eq:dual_dynamics}), the equilibria set $\mathcal{E}$ in (\ref{eq:equilibria_set}), and the map $V:\Delta_\mathcal{P}\times\Delta_\mathcal{D}\rightarrow\mathbb{R}_{\geq0}$ given by
	\begin{equation}
	\label{eq:lyapunov_function}
	V(\mathbf{x}, \boldsymbol{\mu}) = \sum_{j\in\mathcal{S}}\sum_{i\in\mathcal{S}}x_iP_i^j(\mathbf{x}, \boldsymbol{\mu}) + \sum_{l\in\mathcal{C}_e}\sum_{k\in\mathcal{C}_e}\mu_k\Phi_k^l(\mathbf{x}),
	\end{equation}
	where
	\begin{equation*}
	\begin{split}
	P_i^j(\mathbf{x}, \boldsymbol{\mu}) &= \int_0^{f_j^\mu(\mathbf{x},\boldsymbol{\mu}) - f_i^\mu(\mathbf{x},\boldsymbol{\mu})}\rho_j(\tau)d\tau, \quad \forall i,j\in\mathcal{S}\\
	\Phi_k^l(\mathbf{x}) &= \int_0^{g_l(\mathbf{x}) - g_k(\mathbf{x})}\phi_l(\tau)d\tau, \quad \forall k,l\in\mathcal{C}_e.
	\end{split}
	\end{equation*}
	Then, $V(\mathbf{x}, \boldsymbol{\mu})\geq0$ for all $(\mathbf{x}, \boldsymbol{\mu})\in\Delta_\mathcal{P}\times\Delta_\mathcal{D}$, and $V(\mathbf{x}^*, \boldsymbol{\mu}^*)=0$ if and only if $(\mathbf{x}^*, \boldsymbol{\mu}^*)\in\mathcal{E}$.
\end{lem}

\begin{proof}
	First, note that $\mathbb{R}_{\geq0}$ is indeed the codomain of (\ref{eq:lyapunov_function}). This fact follows from $\mathbf{x}\in\Delta_{\mathcal{P}}$, $\boldsymbol{\mu}\in\Delta_{\mathcal{D}}$, $P_i^j(\cdot, \cdot)\geq0$ for all $i,j\in\mathcal{S}$, and $\Phi_k^l(\cdot)\geq0$ for all $k,l\in\mathcal{C}_e$. To prove the second claim, we consider the sufficient and necessary cases separately.
	
	(Sufficiency) If $(\mathbf{x}^*, \boldsymbol{\mu}^*)\in\mathcal{E}$, then $\mathbf{x}^*\in\text{NE}\left(\mathbf{f}^\mu, \boldsymbol{\mu}^*\right)$ and $\boldsymbol{\mu}^*\in\text{NE}\left(\mathbf{g}, \mathbf{x}^*\right)$. In consequence, $x_i^*>0\implies f_i^\mu(\mathbf{x}^*, \boldsymbol{\mu}^*)\geq f_j^\mu(\mathbf{x}^*, \boldsymbol{\mu}^*) \implies P_i^j(\mathbf{x}^*, \boldsymbol{\mu}^*)=0$, for all $i,j\in\mathcal{S}$. Similarly, $\mu_k^*>0\implies g_k(\mathbf{x}^*) \geq g_l(\mathbf{x}^*) \implies \Phi_k^l(\mathbf{x}^*)=0$, for all $k,l\in\mathcal{C}_e$. Therefore, $V(\mathbf{x}^*, \boldsymbol{\mu}^*)=0$, for all $(\mathbf{x}^*, \boldsymbol{\mu}^*)\in\mathcal{E}$.
	
	(Necessity) Consider a point $(\mathbf{x}^*, \boldsymbol{\mu})\in\Delta_{\mathcal{P}}\times\Delta_{\mathcal{D}}$. Let $\mathbf{x}^*\in\text{NE}\left(\mathbf{f}^\mu, \boldsymbol{\mu}\right)$ but suppose $\boldsymbol{\mu}\notin\text{NE}\left(\mathbf{g}, \mathbf{x}^*\right)$. Then, there exists some $w,z\in\mathcal{C}_e$ such that $\mu_w>0$ and $g_w(\mathbf{x}^*)<g_z(\mathbf{x}^*)$. In consequence, $\mu_w\Phi_w^z(\mathbf{x}^*)>0$ and $V(\mathbf{x}^*, \boldsymbol{\mu})>0$. Now, consider a point $(\mathbf{x}, \boldsymbol{\mu}^*)\in\Delta_{\mathcal{P}}\times\Delta_{\mathcal{D}}$. Let $\boldsymbol{\mu}^*\in\text{NE}\left(\mathbf{g}, \mathbf{x}\right)$ but suppose $\mathbf{x}\notin\text{NE}\left(\mathbf{f}^\mu, \boldsymbol{\mu}^*\right)$. Then, there exists some $v,y\in\mathcal{S}$ such that $x_v>0$ and $f_v^\mu(\mathbf{x}, \boldsymbol{\mu}^*) < f_y^\mu(\mathbf{x}, \boldsymbol{\mu}^*)$. Therefore, $x_vP_v^y(\mathbf{x}, \boldsymbol{\mu}^*)>0$ and $V(\mathbf{x}, \boldsymbol{\mu}^*)>0$. Hence, any $(\mathbf{x}, \boldsymbol{\mu})\in\left(\Delta_{\mathcal{P}}\times\Delta_{\mathcal{D}}\right)\setminus\mathcal{E}$ implies that $V(\mathbf{x}, \boldsymbol{\mu})>0$. %\hfill $\blacksquare$
\end{proof}

Using Lemma \ref{lem:lyapunov_function}, we now state the following theorem that provides sufficient conditions to guarantee the asymptotic stability of the equilibria set $\mathcal{E}$ under the considered primal-dual dynamics.

\setcounter{thm}{2}
\begin{thm}
	\label{thm:lyapunov_stability_potential_games}
	Let Assumptions \ref{assump:full_potential_game} and \ref{assump:convex_constraints} hold. Consider the primal-dual system (\ref{eq:primal_dynamics})-(\ref{eq:dual_dynamics}), and the equilibria set $\mathcal{E}$ in (\ref{eq:equilibria_set}). Moreover, assume that $\mathcal{E}$ is nonempty. Then, $\mathcal{E}$ is asymptotically stable under the considered dynamics.
\end{thm}

\begin{proof}
	Notice that, due to the local Lipschitz continuity of $\rho_j(\cdot)$ and $\phi_l(\cdot)$, for all $j\in\mathcal{S}$ and all $l\in\mathcal{C}_e$, and the compactness of $\Delta_{\mathcal{P}}$ and $\Delta_\mathcal{D}$, it follows that the primal-dual dynamics are locally Lipschitz continuous. Moreover, note that the set $\mathcal{E}$ is compact. To see this, observe that $\mathcal{E}\subseteq\Delta_{\mathcal{P}}\times\Delta_{\mathcal{D}}$ is bounded, and that $\mathcal{E}$ is closed because it is the preimage of the closed set $\{0\}$ under the continuous map provided by (\ref{eq:lyapunov_function}). In consequence, the primal-dual dynamics can be investigated using standard Lyapunov stability theory \cite[Corollary 4.7]{haddad2008}. From Lemma \ref{lem:lyapunov_function}, it follows that (\ref{eq:lyapunov_function}) is a valid Lyapunov function candidate. Hence, we proceed to analyze its derivatives. For such, let $f_i\triangleq f_i(\mathbf{x})$, $f_i^\mu\triangleq f_i^\mu(\mathbf{x}, \boldsymbol{\mu})$, $g_k\triangleq g_k(\mathbf{x})$, $P_i^j\triangleq P_i^j(\mathbf{x}, \boldsymbol{\mu})$, $\Phi_k^l\triangleq \Phi_k^l(\mathbf{x})$, $\rho_i^j\triangleq \rho_i^j(\mathbf{x}, \boldsymbol{\mu})$, and $\phi_k^l\triangleq\phi_k^l(\mathbf{x})$, and note that
	\begin{equation*}
	\begin{split}
	&\frac{\partial V(\mathbf{x}, \boldsymbol{\mu})}{\partial x_y} = \sum_{j\in\mathcal{S}}P_y^j + \sum_{j\in\mathcal{S}}\sum_{i\in\mathcal{S}}x_i\frac{\partial P_i^j}{\partial x_y} + \sum_{l\in\mathcal{C}_e}\sum_{k\in\mathcal{C}_e}\mu_k\frac{\partial \Phi_k^l}{\partial x_y}\\
	&\frac{\partial V(\mathbf{x}, \boldsymbol{\mu})}{\partial \mu_z} = \sum_{j\in\mathcal{S}}\sum_{i\in\mathcal{S}}x_i\frac{\partial P_i^j}{\partial \mu_z} + \sum_{l\in\mathcal{C}_e}\Phi_z^l,
	\end{split}
	\end{equation*}
	for all $y\in\mathcal{S}$ and all $z\in\mathcal{C}_e$. Moreover, observe that
	\begin{equation*}
	\begin{split}
	\sum_{j\in\mathcal{S}}\sum_{i\in\mathcal{S}}x_i\frac{\partial P_i^j}{\partial x_y} &= \sum_{j\in\mathcal{S}}\sum_{i\in\mathcal{S}}x_i\rho_i^j\left(\frac{\partial f_j^\mu}{\partial x_y} - \frac{\partial f_i^\mu}{\partial x_y}\right)\\
	&= \sum_{j\in\mathcal{S}}\sum_{i\in\mathcal{S}}x_i\rho_i^j\frac{\partial f_j^\mu}{\partial x_y} - \sum_{j\in\mathcal{S}}\sum_{i\in\mathcal{S}}x_i\rho_i^j\frac{\partial f_i^\mu}{\partial x_y}\\
	&= \sum_{j\in\mathcal{S}}\sum_{i\in\mathcal{S}}x_i\rho_i^j\frac{\partial f_j^\mu}{\partial x_y} - \sum_{j\in\mathcal{S}}\sum_{i\in\mathcal{S}}x_j\rho_j^i\frac{\partial f_j^\mu}{\partial x_y}\\
	&= \sum_{j\in\mathcal{S}}\sum_{i\in\mathcal{S}}\left(x_i\rho_i^j - x_j\rho_j^i\right)\frac{\partial f_j^\mu}{\partial x_y}\\
	&= \sum_{j\in\mathcal{S}}\dot{x}_j\frac{\partial f_j^\mu}{\partial x_y} \quad (\text{using (\ref{eq:primal_dot})}).
	\end{split}
	\end{equation*}
	Similarly,
	\begin{equation*}
	\begin{split}
	\sum_{l\in\mathcal{C}_e}\sum_{k\in\mathcal{C}_e}\mu_k\frac{\partial \Phi_k^l}{\partial x_y} &= \sum_{l\in\mathcal{C}_e}\sum_{k\in\mathcal{C}_e}\mu_k\phi_k^l\left(\frac{\partial g_l}{\partial x_y} - \frac{\partial g_k}{\partial x_y}\right)\\
	&= \sum_{l\in\mathcal{C}_e}\sum_{k\in\mathcal{C}_e}\left(\mu_k\phi_k^l - \mu_l\phi_l^k\right)\frac{\partial g_l}{\partial x_y}\\
	&= \sum_{l\in\mathcal{C}_e}\dot{\mu}_l\frac{\partial g_l}{\partial x_y}  \quad (\text{using (\ref{eq:dual_dot})}),
	\end{split}
	\end{equation*}
	and
	\begin{equation*}
	\begin{split}
	\sum_{j\in\mathcal{S}}\sum_{i\in\mathcal{S}}x_i\frac{\partial P_i^j}{\partial \mu_z} &= \sum_{j\in\mathcal{S}}\sum_{i\in\mathcal{S}}x_i\rho_i^j\left(\frac{\partial f_j^\mu}{\partial \mu_z} - \frac{\partial f_i^\mu}{\partial \mu_z}\right)\\
	&= \sum_{j\in\mathcal{S}}\sum_{i\in\mathcal{S}}\left(x_i\rho_i^j - x_j\rho_j^i\right)\frac{\partial f_j^\mu}{\partial \mu_z}\\
	&= \sum_{j\in\mathcal{S}}\dot{x}_j\frac{\partial f_j^\mu}{\partial \mu_z} \quad (\text{using (\ref{eq:primal_dot})})\\
	&= -\sum_{j\in\mathcal{S}}\dot{x}_j\frac{\partial g_z}{\partial x_j}  \quad (\text{using (\ref{eq:primal_dual_game})}).
	\end{split}
	\end{equation*}
	By defining the vectors $\boldsymbol{\Gamma}_P \triangleq \left[\sum_{j\in\mathcal{S}}P_y^j\right]\in\mathbb{R}_{\geq0}^n$ and $\boldsymbol{\Gamma}_\Phi \triangleq \left[\sum_{l\in\mathcal{C}_e}\Phi_z^l\right]\in\mathbb{R}_{\geq0}^{q+1}$, it follows that
	\begin{equation*}
	\begin{split}
	&\nabla_\mathbf{x}V(\mathbf{x}, \boldsymbol{\mu}) = \boldsymbol{\Gamma}_P + \left(\text{D}\mathbf{f}^\mu\right)^\top\mathbf{\dot{x}} + \left(\text{D}\mathbf{g}\right)^\top\mathbf{\dot{\boldsymbol{\mu}}}\\
	&\nabla_{\boldsymbol{\mu}}V(\mathbf{x}, \boldsymbol{\mu}) = \boldsymbol{\Gamma}_\Phi - \text{D}\mathbf{g}\mathbf{\dot{x}},
	\end{split}
	\end{equation*}
	where $\text{D}\mathbf{f}^\mu\in\mathbb{R}^{n\times n}$ is the Jacobian matrix of $\mathbf{f}^\mu(\cdot, \boldsymbol{\mu})$ and is evaluated at $\left(\mathbf{x}(t), \boldsymbol{\mu}(t)\right)$; and $\text{D}\mathbf{g}\in\mathbb{R}^{(q+1)\times n}$ is the Jacobian matrix of $\mathbf{g}(\cdot)$ and is evaluated at $\mathbf{x}(t)$. Therefore, setting $\nabla_\mathbf{x}V \triangleq \nabla_\mathbf{x}V(\mathbf{x}, \boldsymbol{\mu})$ and $\nabla_{\boldsymbol{\mu}}V \triangleq \nabla_{\boldsymbol{\mu}}V(\mathbf{x}, \boldsymbol{\mu})$, it follows that
	\begin{equation*}
	\begin{split}
	&\left[\nabla_\mathbf{x}V^\top, \nabla_{\boldsymbol{\mu}}V^\top\right]\left[\begin{array}{c}\mathbf{\dot{x}}\\
	\mathbf{\dot{\boldsymbol{\mu}}}\end{array}\right]\\
	&\quad\,\,= \boldsymbol{\Gamma}_P^\top\mathbf{\dot{x}} + \mathbf{\dot{x}}^\top\text{D}\mathbf{f}^\mu\mathbf{\dot{x}} + \mathbf{\dot{\boldsymbol{\mu}}}^\top\text{D}\mathbf{g}\mathbf{\dot{x}} + \boldsymbol{\Gamma}_\Phi^\top\mathbf{\dot{\boldsymbol{\mu}}} - \mathbf{\dot{x}}^\top\left(\text{D}\mathbf{g}\right)^\top\mathbf{\dot{\boldsymbol{\mu}}}\\
	&\quad\,\,= \boldsymbol{\Gamma}_P^\top\mathbf{\dot{x}} + \mathbf{\dot{x}}^\top\text{D}\mathbf{f}^\mu\mathbf{\dot{x}} + \boldsymbol{\Gamma}_\Phi^\top\mathbf{\dot{\boldsymbol{\mu}}},
	\end{split}
	\end{equation*}
	where $\mathbf{\dot{\boldsymbol{\mu}}}^\top\text{D}\mathbf{g}\mathbf{\dot{x}} = \left(\mathbf{\dot{\boldsymbol{\mu}}}^\top\text{D}\mathbf{g}\mathbf{\dot{x}}\right)^\top = \mathbf{\dot{x}}^\top\left(\text{D}\mathbf{g}\right)^\top\mathbf{\dot{\boldsymbol{\mu}}}$ since it is a scalar. From the concavity of $p(\cdot)$ (c.f., Assumption \ref{assump:full_potential_game}) and the convexity of $g_k(\cdot)$ for all $k\in\mathcal{C}_e$ (c.f., Assumption \ref{assump:convex_constraints} and recall that $g_0(\cdot)=0$), it follows that $\mathbf{\dot{x}}^\top\text{D}\mathbf{f}^\mu\mathbf{\dot{x}}\leq0$ for all $t\geq0$, and $(\mathbf{x}, \boldsymbol{\mu})\in\mathcal{E}$ implies that $\mathbf{\dot{x}}^\top\text{D}\mathbf{f}^\mu\mathbf{\dot{x}}=0$ since $\mathbf{\dot{x}}=\mathbf{0}$ at $\mathcal{E}$ (c.f., Lemma \ref{lem:primal_dual_nash_stationarity}). Hence, we proceed to analyze the other two terms. Note that,
	\begin{equation*}
	\begin{split}
	\boldsymbol{\Gamma}_P^\top\mathbf{\dot{x}} &= \sum_{y\in\mathcal{S}}\dot{x}_y\sum_{j\in\mathcal{S}}P_y^j\\
	&= \sum_{y\in\mathcal{S}}\sum_{i\in\mathcal{S}}\left(x_i\rho_i^y - x_y\rho_y^i\right)\sum_{j\in\mathcal{S}}P_y^j \quad (\text{using (\ref{eq:primal_dot})})\\
	&= \sum_{y\in\mathcal{S}}\sum_{i\in\mathcal{S}}x_i\rho_i^y\sum_{j\in\mathcal{S}}P_y^j - \sum_{y\in\mathcal{S}}\sum_{i\in\mathcal{S}}x_y\rho_y^i\sum_{j\in\mathcal{S}}P_y^j\\
	&= \sum_{y\in\mathcal{S}}\sum_{i\in\mathcal{S}}x_i\rho_i^y\sum_{j\in\mathcal{S}}P_y^j - \sum_{y\in\mathcal{S}}\sum_{i\in\mathcal{S}}x_i\rho_i^y\sum_{j\in\mathcal{S}}P_i^j\\
	&= \sum_{y\in\mathcal{S}}\sum_{i\in\mathcal{S}}x_i\rho_i^y\sum_{j\in\mathcal{S}}\left(P_y^j - P_i^j\right).\\
	\end{split}
	\end{equation*}
	Similarly,
	\begin{equation*}
	\begin{split}
	\boldsymbol{\Gamma}_\Phi^\top\mathbf{\dot{\boldsymbol{\mu}}} &= \sum_{z\in\mathcal{C}_e}\dot{\mu}_z\sum_{l\in\mathcal{C}_e}\Phi_z^l\\
	&= \sum_{z\in\mathcal{C}_e}\sum_{k\in\mathcal{C}_e}\left(\mu_k\phi_k^z - \mu_z\phi_z^k\right)\sum_{l\in\mathcal{C}_e}\Phi_z^l \quad (\text{using (\ref{eq:dual_dot})})\\
	&= \sum_{z\in\mathcal{C}_e}\sum_{k\in\mathcal{C}_e}\mu_k\phi_k^z\sum_{l\in\mathcal{C}_e}\left(\Phi_z^l - \Phi_k^l\right).\\
	\end{split}
	\end{equation*}
	Here, note that $f_y^\mu\leq f_i^\mu\implies \rho_i^y=0$ and $g_z \leq g_k \implies \phi_k^z=0$. Thus, it suffices to analyze the relevant cases where $f_y^\mu>f_i^\mu$ and $g_z>g_k$. In particular, notice that
	\begin{equation*}
	\begin{split}
	&f_j^\mu \geq f_y^\mu > f_i^\mu \implies P_y^j - P_i^j < 0\\
	&f_y^\mu > f_j^\mu > f_i^\mu \implies P_y^j - P_i^j = 0 - P_i^j < 0\\
	&f_y^\mu > f_i^\mu \geq f_j^\mu \implies P_y^j - P_i^j = 0 - 0 = 0,\\
	\text{and}\quad &g_l \geq g_z > g_k \implies \Phi_z^l - \Phi_k^l < 0\\
	&g_z > g_l > g_k \implies \Phi_z^l - \Phi_k^l = 0 - \Phi_k^l < 0\\
	&g_z > g_k \geq g_l \implies \Phi_z^l - \Phi_k^l = 0 - 0 = 0.\\
	\end{split}
	\end{equation*}
	Therefore, $\boldsymbol{\Gamma}_P^\top\mathbf{\dot{x}}\leq0$ and $\boldsymbol{\Gamma}_\Phi^\top\mathbf{\dot{\boldsymbol{\mu}}}\leq0$ for all $t\geq0$. Hence, $\mathcal{E}$ is stable in the sense of Lyapunov. Furthermore, using (\ref{eq:conditions_protocols_primal}), (\ref{eq:primal_memoryless}), and the fact that $\mathbf{x}\in\text{NE}\left(\mathbf{f}^\mu, \boldsymbol{\mu}\right)$ implies that $x_i>0\implies f_i^\mu = \max_{j\in\mathcal{S}}f_j^\mu$, for all $i\in\mathcal{S}$, it can be shown that $\boldsymbol{\Gamma}_P^\top\mathbf{\dot{x}}=0\iff\mathbf{x}\in\text{NE}\left(\mathbf{f}^\mu, \boldsymbol{\mu}\right)$ \cite[Theorem 7.1]{hofbauer2009stable}. Since the dual dynamics have the same form as the primal dynamics, using the same arguments it follows that $\boldsymbol{\Gamma}_\Phi^\top\mathbf{\dot{\boldsymbol{\mu}}}=0\iff\boldsymbol{\mu}\in\text{NE}\left(\mathbf{g}, \mathbf{x}\right)$. In consequence, the equilibria set $\mathcal{E}$ is asymptotically stable. %\hfill $\blacksquare$	
\end{proof}

\setcounter{thm}{6}
\begin{rem}
	\label{rem:theorems_2_and_3}
	Notice that if all the conditions of Theorem \ref{thm:equilibria_set_properties} hold, then the result of Theorem \ref{thm:lyapunov_stability_potential_games} follows immediately. Moreover, in such case the set $\mathcal{E}$ is not only asymptotically stable, but is also the set of primal-dual optimal solutions of the convex optimization problem in (\ref{eq:primal_problem}).
\end{rem}

Note that, under some different assumptions, the results of Theorem \ref{thm:lyapunov_stability_potential_games} are valid for other classes of games that might not necessarily be full-potential games. We highlight such result in the following corollary.

\setcounter{thm}{0}
\begin{cor}
	\label{cor:stable_games}
	Let Assumption \ref{assump:convex_constraints} hold. Consider the primal-dual system (\ref{eq:primal_dynamics})-(\ref{eq:dual_dynamics}), and the equilibria set $\mathcal{E}$ in (\ref{eq:equilibria_set}). Moreover, assume that $\mathcal{E}$ is nonempty and that the primal game $\mathbf{f}(\cdot)$ is continuously differentiable and satisfies that $\mathbf{\dot{x}}^\top\text{\normalfont D}\mathbf{f}\mathbf{\dot{x}}\leq0$, for all $t\geq0$, where $\text{\normalfont D}\mathbf{f}\in\mathbb{R}^{n\times n}$ is the Jacobian matrix of $\mathbf{f}(\cdot)$ and is evaluated at $\mathbf{x}(t)$; and $\mathbf{\dot{x}}=[\dot{x}_i]\in\mathbb{R}^n$ is determined by (\ref{eq:primal_dot}). Then, $\mathcal{E}$ is asymptotically stable under the considered dynamics. 
\end{cor}

\begin{proof}
	Under the additional assumptions, and noting that $\mathbf{\dot{x}}^\top\text{\normalfont D}\mathbf{f}\mathbf{\dot{x}}\leq0$ implies that $\mathbf{\dot{x}}^\top\text{\normalfont D}\mathbf{f}^\mu\mathbf{\dot{x}}\leq0$, the proof is virtually identical to the proof of Theorem \ref{thm:lyapunov_stability_potential_games}. %\hfill $\blacksquare$
\end{proof}

\setcounter{thm}{7}
\begin{rem}
	Although the interpretation of the proposed primal-dual dynamics fits more naturally under the scope of full-potential games, Corollary \ref{cor:stable_games} generalizes our previous stability results to more general classes of primal games. In particular, one class of games that satisfies the condition $\mathbf{\dot{x}}^\top\text{\normalfont D}\mathbf{f}\mathbf{\dot{x}}\leq0$ is the class of stable/contractive games \cite{hofbauer2009stable}, \cite{parkCDC2019}.
\end{rem}

\section{Numerical experiments}
\label{sec:numerical_experiments}
In this section, we provide some numerical experiments to illustrate the formulation of the proposed primal-dual system regarding some classical games that are extended to consider constraints. It is worth to highlight that such constraints cannot be considered using conventional population dynamics. Without loss of generality,  we use $\max(\cdot, 0)$ for the functions $\rho_j(\cdot)$ and $\phi_l(\cdot)$, for all $j\in\mathcal{S}$ and all $l\in\mathcal{C}_e$. Note that such revision protocols lead to the popular Smith dynamics \cite{smith1984}. Moreover, for all the numerical integrations we use a step size of $0.01$s.

\subsection{A constrained congestion game}
Consider a population of players that seek to travel from point A to point B under the topology depicted in Fig. \ref{fig:congestion_topology}. There are eight roads, $\{r_k\}_{k=1}^8$, that together provide four possible strategies to travel from A to B. They are, $s_1=\{r_1, r_2\}$, $s_2=\{r_8, r_7, r_3, r_2\}$, $s_3=\{r_8, r_7, r_5, r_4\}$, and $s_4=\{r_8, r_6, r_4\}$. In contrast with classical congestion games \cite{sandholm2010}, here we assume that each road has a maximum usage level denoted as $\overline{u_k}$, for all $k=1,2,\dots,8$. For our experiment, we set $\overline{u_1}=\overline{u_3}=\overline{u_5}=\overline{u_6}=0.4$, $\overline{u_2}=\overline{u_4}=\overline{u_7}=0.6$, and $\overline{u_8}=0.9$, and we assume that $m_\mathcal{P}=1$. Furthermore, we assume that each road has a linear congestion cost given by $b_ku_k$, where $b_k\in\mathbb{R}_{>0}$ and $u_k\in\mathbb{R}_{\geq0}$ are the weight of congestion and usage level of the the $k$-th road, respectively. For simplicity, and to resemble the order of magnitude of the parameters of \cite[Example 3.1.6]{sandholm2010}, we (randomly) set $\mathbf{b}=[b_k]=[15,16,11,13,13,5,17,18]^\top$. The potential function for such a congestion game is thus given by
\begin{equation*}
\begin{split}
p(\mathbf{x}) = &-\frac{b_1}{2}x_1^2 - \frac{b_2}{2}(x_1 + x_2)^2 - \frac{b_3}{2}x_2^2 - \frac{b_4}{2}(x_3 + x_4)^2\\
&-\frac{b_5}{2}x_3^2 -\frac{b_6}{2}x_4^2 - \frac{b_7}{2}(x_2 + x_3)^2\\
&- \frac{b_8}{2}(x_2 + x_3 + x_4)^2,
\end{split}
\end{equation*}
which satisfies Assumption \ref{assump:full_potential_game}. Moreover, the maximum usage levels lead to the constraints $x_1 \leq \overline{u_1}$; $x_2 \leq \overline{u_3}$; $x_3 \leq \overline{u_5}$; $x_4 \leq \overline{u_6}$; $x_1+x_2 \leq \overline{u_2}$; $x_2+x_3 \leq \overline{u_7}$; $x_3+x_4 \leq \overline{u_4}$; and $x_2+x_3+x_4 \leq \overline{u_8}$. Thus, Assumption \ref{assump:convex_constraints} holds. Consequently, the primal-dual game is given by $f_1^\mu(\mathbf{x}, \boldsymbol{\mu}) = f_1(\mathbf{x}) - \mu_1 -\mu_5$; $f_2^\mu(\mathbf{x}, \boldsymbol{\mu}) = f_2(\mathbf{x}) - \mu_2 - \mu_5 - \mu_6 - \mu_8$; $f_3^\mu(\mathbf{x}, \boldsymbol{\mu}) = f_3(\mathbf{x}) - \mu_3 - \mu_6 - \mu_7 - \mu_8$; and $f_4^\mu(\mathbf{x}, \boldsymbol{\mu}) = f_4(\mathbf{x}) - \mu_4 - \mu_7 - \mu_8$,
where $f_i(\cdot) = \partial p(\cdot)/\partial x_i$, for all $i\in\mathcal{S}$. Notice that, with $\mathbf{\tilde{x}}=[0.25,0.25,0.25,0.25]^\top$, it is verified that Assumption \ref{assump:slater} holds. Moreover, observe that $p^*\leq0$, $p(\mathbf{\tilde{x}})=-12.1875$, and $\min_{k\in\mathcal{C}}|g_k(\mathbf{\tilde{x}})|=0.1$. Hence, setting $m_\mathcal{D}=122$ satisfies the condition (\ref{eq:condition_dual_population_mass}) of Theorem \ref{thm:equilibria_set_properties}. For such, we set $\mu_0(0)=122$ and $\mu_k(0)=0$ for all $k\in\mathcal{C}$, so that $\boldsymbol{\mu}(0)\in\Delta_{\mathcal{D}}$. Regarding the primal population, on the other hand, we randomly sample $\mathbf{x}(0)$ from $\Delta_{\mathcal{P}}$. Therefore, the Standing Assumption 1 holds. In consequence, Theorems \ref{thm:equilibria_set_properties} and \ref{thm:lyapunov_stability_potential_games} hold and the equilibria set $\mathcal{E}$ is asymptotically stable and is the primal-dual optimal set of the corresponding problem in (\ref{eq:primal_problem}) (c.f., Remark \ref{rem:theorems_2_and_3}). As illustration, Fig. \ref{fig:results_congestion} depicts the temporal evolution of the primal-dual system. Note that convergence to a fixed point is achieved, and, in fact, such fixed point belongs to $\mathcal{E}$ (this is verified with the aid of a convex optimization solver as shown in Fig. \ref{fig:results_congestion}).

\begin{figure}
	\centering
	\includegraphics[width=0.27\textwidth]{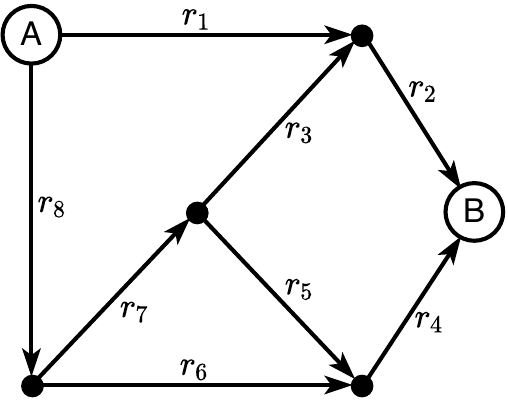}
	\caption{\, Considered topology for the congestion game.}
	\label{fig:congestion_topology}
\end{figure}

\begin{figure}
	\centering
	\includegraphics[width=0.45\textwidth]{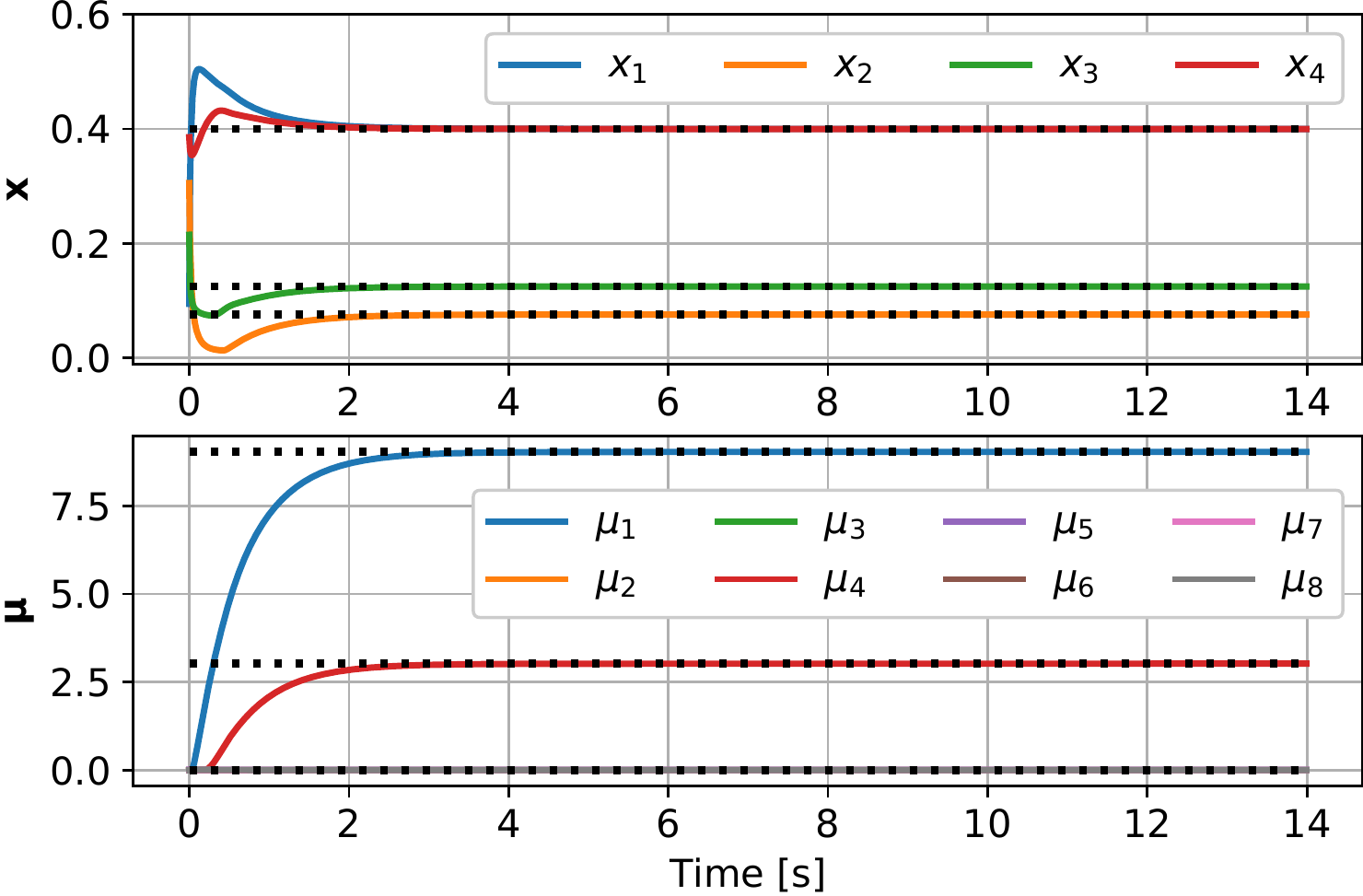}
	\caption{\, Evolution of the primal-dual system for the considered congestion game. The dotted black lines depict the optimal values of the primal and dual variables of the underlying optimization problem. Such values are obtained using CVXPY \cite{diamond2016cvxpy}. Note that we are not plotting $\mu_0$ due to its scale (recall that $\mu_0(t)=m_\mathcal{D} - \sum_{k\in\mathcal{C}}\mu_k(t)$ for all $t\geq0$).} 
	\label{fig:results_congestion}
\end{figure}

\subsection{Good Rock-Paper-Scissors with constraints}
Even though most of the developed analyses have been oriented to full-potential games (c.f., Assumption \ref{assump:full_potential_game}), as shown in Corollary \ref{cor:stable_games}, the developed approach is still applicable to more general classes of games, e.g., stable/contractive games. As illustration, consider an instance of the classical Good Rock Paper Scissor (RPS) game with a payoff matrix and corresponding fitness functions given by
\begin{equation*}
\mathbf{A} = \left[\begin{array}{ccc}0 & -1 & 2\\
2 & 0 & -1\\
-1 & 2 & 0\end{array}\right], \quad \begin{array}{c}f_1(\mathbf{x}) = 2x_3-x_2,\\
f_2(\mathbf{x}) = 2x_1 -x_3,\\
f_3(\mathbf{x}) = 2x_2-x_1,\end{array}
\end{equation*}
where we identify the three strategies as $1\rightarrow\text{Rock}$, $2\rightarrow\text{Paper}$, and $3\rightarrow\text{Scissors}$. Moreover, we consider an unitary mass of agents, i.e., $m_\mathcal{P}=1$, and the coupled convex constraint $x_1^2 + x_2^2 \leq0.1$ (note that such constraint rules out the unconstrained Nash equilibrium $x_1^*=x_2^*=x_3^*=1/3$). Clearly, the RPS game is not a full-potential game. Thus, Assumption \ref{assump:full_potential_game} is not satisfied. Nevertheless, it does satisfies the condition $\mathbf{\dot{x}}^\top\text{\normalfont D}\mathbf{f}\mathbf{\dot{x}}\leq0$, for all $t\geq0$, and so the result of Corollary \ref{cor:stable_games} follows. In this case, the primal-dual game is defined as $f_1^\mu(\boldsymbol{\mu}, \mathbf{x}) = 2x_3-x_2 - 2x_1\mu_1$; $f_2^\mu(\boldsymbol{\mu}, \mathbf{x}) = 2x_1-x_3 - 2x_2\mu_1$; and $f_3^\mu(\boldsymbol{\mu}, \mathbf{x}) = 2x_2-x_1$. Moreover, for our experiment we set $\mathbf{x}(0)=\left[1/3, 1/3, 1/3\right]^\top$ and $\boldsymbol{\mu}(0)=[4,0]^\top$, and so $m_\mathcal{D}=4$. As illustration, Fig. \ref{fig:stable_game} depicts the evolution of the primal-dual dynamics under such game. Notice that regardless of the non-full-potential nature of the game, the primal-dual dynamics do converge to $\mathcal{E}$, which in this case is feasible with respect to $\mathcal{X}$. However, it remains as an open topic to formally characterize the conditions for the feasibility of $\mathcal{E}$ under non-full-potential games, e.g., contractive or weighted-contractive games \cite{arcak2020dissipativity}.

\begin{figure}
	\centering
	\includegraphics[width=0.45\textwidth]{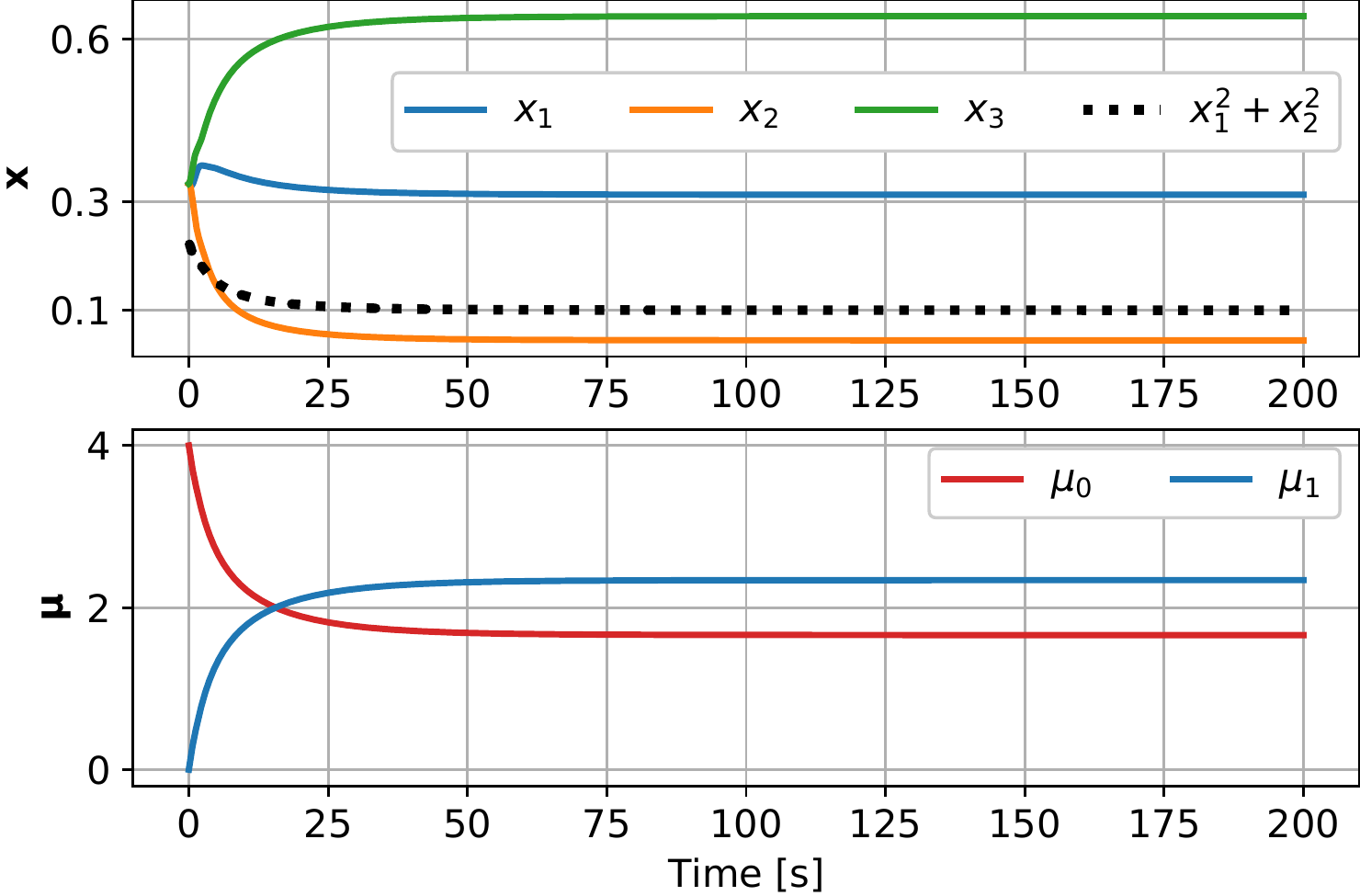}
	\caption{\, Evolution of the primal-dual system under the considered RPS game. Here, $\mathbf{x}^*\approx[0.313, 0.044, 0.643]^\top$.}
	\label{fig:stable_game}
\end{figure}

\section{Concluding remarks}
\label{sec:concluding_remarks}
In this paper, we have proposed and analyzed a class of primal-dual evolutionary dynamics that can be applied to study constrained population games considering general convex inequality constraints. We have provided sufficient conditions to guarantee the asymptotic stability of the equilibria set of the proposed dynamics when applied both to full-potential and non-full-potential games. Moreover, for the context of full-potential games, we have deduced sufficient conditions to guarantee that the equilibria set of the dynamics is feasible with respect to the considered convex constraints. Future work should focus on the extension of Theorem \ref{thm:equilibria_set_properties} to more general classes of games, e.g., contractive or weighted-contractive games, and to the study of the dynamic satisfaction of the considered convex constraints.

\bibliographystyle{plain} 
\bibliography{bibfile}

\begin{thebibliography}{10}

\bibitem{ALGHUNAIM2020109003}
Sulaiman~A. Alghunaim and Ali~H. Sayed.
\newblock Linear convergence of primal–dual gradient methods and their
  performance in distributed optimization.
\newblock {\em Automatica}, 117:109003, 2020.

\bibitem{arcak2020dissipativity}
Murat Arcak and Nuno~C. Martins.
\newblock Dissipativity tools for convergence to nash equilibria in population
  games.
\newblock {\em arXiv preprint arXiv:2005.03797}, 2020.

\bibitem{barreiroauto2016}
Julian Barreiro-Gomez, Nicanor Quijano, and Carlos Ocampo-Martinez.
\newblock Constrained distributed optimization: A population dynamics approach.
\newblock {\em Automatica}, 69:101--116, 2016.

\bibitem{barreiroauto2018}
Julian Barreiro-Gomez and Hamidou Tembine.
\newblock Constrained evolutionary games by using a mixture of imitation
  dynamics.
\newblock {\em Automatica}, 97:254 -- 262, 2018.

\bibitem{bertsekas2009convex}
Dimitri~P. Bertsekas.
\newblock {\em Convex optimization theory}.
\newblock Athena Scientific Belmont, 2009.

\bibitem{boyd2011distributed}
Stephen Boyd, Neal Parikh, and Eric Chu.
\newblock {\em Distributed optimization and statistical learning via the
  alternating direction method of multipliers}.
\newblock Now Publishers Inc, 2011.

\bibitem{diamond2016cvxpy}
Steven Diamond and Stephen Boyd.
\newblock {CVXPY}: {A} {P}ython-embedded modeling language for convex
  optimization.
\newblock {\em Journal of Machine Learning Research}, 17(83):1--5, 2016.

\bibitem{fox2013population}
Michael~J. Fox and Jeff~S. Shamma.
\newblock Population games, stable games, and passivity.
\newblock {\em Games}, 4(4):561--583, 2013.

\bibitem{govaert2019}
Alain {Govaert}, Carlo {Cenedese}, Sergio {Grammatico}, and Ming {Cao}.
\newblock Relative best response dynamics in finite and convex network games.
\newblock In {\em Proceedings of the 2019 IEEE 58th Conference on Decision and
  Control (CDC)}, pages 3134--3139, 2019.

\bibitem{grammatico2018}
Sergio {Grammatico}.
\newblock Proximal dynamics in multiagent network games.
\newblock {\em IEEE Transactions on Control of Network Systems},
  5(4):1707--1716, 2018.

\bibitem{haddad2008}
Wassim~M. Haddad and VijaySekhar Chellaboina.
\newblock {\em Nonlinear dynamical systems and control: a Lyapunov-based
  approach}.
\newblock Princeton University Press, 2008.

\bibitem{hofbauer2009stable}
Josef Hofbauer and William~H Sandholm.
\newblock Stable games and their dynamics.
\newblock {\em Journal of Economic Theory}, 144(4):1665--1693, 2009.

\bibitem{hofbauer1998evolutionary}
Josef Hofbauer and Karl Sigmund.
\newblock {\em Evolutionary games and population dynamics}.
\newblock Cambridge University Press, 1998.

\bibitem{LEI2016110}
Jinlong Lei, Han-Fu Chen, and Hai-Tao Fang.
\newblock Primal–dual algorithm for distributed constrained optimization.
\newblock {\em Systems $\&$ Control Letters}, 96:110 -- 117, 2016.

\bibitem{liang2020}
Shu {Liang}, Le~Yi {Wang}, and George {Yin}.
\newblock Distributed smooth convex optimization with coupled constraints.
\newblock {\em IEEE Transactions on Automatic Control}, 65(1):347--353, 2020.

\bibitem{parkCDC2019}
Shinkyu Park, Nuno~C. Martins, and Jeff~S. Shamma.
\newblock From population games to payoff dynamics models: A passivity-based
  approach.
\newblock In {\em Proceedings of the 58th IEEE Conference on Decision and
  Control (CDC)}, pages 6584--6601, 2019.

\bibitem{parkarxiv2019}
Shinkyu Park, Nuno~C. Martins, and Jeff~S. Shamma.
\newblock Payoff dynamics model and evolutionary dynamics model: Feedback and
  convergence to equilibria.
\newblock {\em arXiv preprint arXiv:1903.02018}, 2019.

\bibitem{park2018CDC}
Shinkyu {Park}, Jeff~S. {Shamma}, and Nuno~C. {Martins}.
\newblock Passivity and evolutionary game dynamics.
\newblock In {\em Proceedings of the 2018 IEEE Conference on Decision and
  Control (CDC)}, pages 3553--3560, 2018.

\bibitem{pashaie2017}
Ashkan Pashaie, Lacra Pavel, and Christopher~J. Damaren.
\newblock A population game approach for dynamic resource allocation problems.
\newblock {\em International Journal of Control}, 90(9):1957--1972, 2017.

\bibitem{quijano2017}
Nicanor Quijano, Carlos Ocampo-Martinez, Julian Barreiro-Gomez, German Obando,
  Andres Pantoja, and Eduardo Mojica-Nava.
\newblock The role of population games and evolutionary dynamics in distributed
  control systems: The advantages of evolutionary game theory.
\newblock {\em IEEE Control Systems Magazine}, 37(1):70--97, 2017.

\bibitem{sandholm2010}
William~H Sandholm.
\newblock {\em Population games and evolutionary dynamics}.
\newblock MIT Press, 2010.

\bibitem{smith1984}
Michael~J. Smith.
\newblock The stability of a dynamic model of traffic assignment: an
  application of a method of {L}yapunov.
\newblock {\em Transportation Science}, 18(3):245--252, 1984.

\bibitem{tembine2010}
Hamidou {Tembine}, Eitan {Altman}, Rachid {El-Azouzi}, and Yezekael {Hayel}.
\newblock Evolutionary games in wireless networks.
\newblock {\em IEEE Transactions on Systems, Man, and Cybernetics, Part B
  (Cybernetics)}, 40(3):634--646, 2010.

\bibitem{zhu2012}
Minghui {Zhu} and Sonia {Martinez}.
\newblock On distributed convex optimization under inequality and equality
  constraints.
\newblock {\em IEEE Transactions on Automatic Control}, 57(1):151--164, 2012.

\bibitem{zino2017}
Lorenzo {Zino}, Giacomo {Como}, and Fabio {Fagnani}.
\newblock On imitation dynamics in potential population games.
\newblock In {\em Proceedings of the 2017 IEEE 56th Annual Conference on
  Decision and Control (CDC)}, pages 757--762, 2017.

\end{thebibliography}

\end{document}